\documentclass[11pt,fleqn]{article}

\usepackage{amsmath,amssymb,amsthm,esint}
\usepackage{geometry}
\usepackage[pdfpagemode=UseNone,pdfstartview=FitH,hypertex]{hyperref}

\newcommand{\abs}[1]{\left|#1\right|}
\newcommand{\bdry}[1]{\partial #1}

\newcommand{\dint}{\ds{\int}}
\newcommand{\dist}[2]{\text{dist}\, (#1,#2)}
\newcommand{\ds}[1]{\displaystyle #1}

\newcommand{\goodchi}{\protect\raisebox{2pt}{$\chi$}}

\newcommand{\id}[1]{id_{#1}}
\newcommand{\incl}{\subset}

\newcommand{\loc}{\text{loc}}
\newcommand{\norm}[2][]{\left\|#2\right\|_{#1}}
\renewcommand{\O}{\text{O}}

\newcommand{\PS}[1]{$(\text{PS})_{#1}$}
\newcommand{\QED}{\mbox{\qedhere}}
\newcommand{\restr}[2]{\left.#1\right|_{#2}}

\newcommand{\set}[1]{\left\{#1\right\}}

\newcommand{\vol}[1]{\left|#1\right|}

\newcommand{\N}{\mathbb N}
\newcommand{\R}{\mathbb R}
\newcommand{\RP}{\R \text{P}}
\newcommand{\Z}{\mathbb Z}

\newcommand{\A}{{\cal A}}
\newcommand{\F}{{\cal F}}
\newcommand{\M}{{\cal M}}

\DeclareMathOperator{\divg}{div}

\newenvironment{properties}[1]{\begin{enumerate}
\renewcommand{\theenumi}{$(#1_\arabic{enumi})$}
\renewcommand{\labelenumi}{$(#1_\arabic{enumi})$}}{\end{enumerate}}

\newtheorem{lemma}{Lemma}[section]
\newtheorem{proposition}[lemma]{Proposition}
\newtheorem{theorem}[lemma]{Theorem}

\theoremstyle{remark}
\newtheorem{remark}[lemma]{Remark}

\numberwithin{equation}{section}

\title{\bf $N$-Laplacian problems with critical Trudinger-Moser nonlinearities}
\author{\bf Yang Yang\thanks{This work was completed while the first-named author was visiting the Department of Mathematical Sciences at the Florida Institute of Technology, and she is grateful for the kind hospitality of the department.
\newline \indent\; Project supported by NSFC-Tian Yuan Special Foundation (No. 11226116), Natural Science Foundation of Jiangsu Province of China for Young Scholars (No. BK2012109), and the China Scholarship Council (No. 201208320435).}\\
School of Science\\
Jiangnan University\\
Wuxi, 214122, China\\
[\bigskipamount]
\bf Kanishka Perera\\
Department of Mathematical Sciences\\
Florida Institute of Technology\\
Melbourne, FL 32901, USA}
\date{}

\begin{document}

\maketitle

\begin{abstract}
We prove existence and multiplicity results for a $N$-Laplacian problem with a critical exponential nonlinearity that is a natural analog of the Brezis-Nirenberg problem for the borderline case of the Sobolev inequality. This extends results in the literature for the semilinear case $N = 2$ to all $N \ge 2$. When $N > 2$ the nonlinear operator $- \Delta_N$ has no linear eigenspaces and hence this extension requires new abstract critical point theorems that are not based on linear subspaces. We prove new abstract results based on the $\Z_2$-cohomological index and a related pseudo-index that are applicable here.
\end{abstract}

\newpage

\section{Introduction and main results}

Elliptic problems with critical nonlinearities have been widely studied in the literature. Let $\Omega$ be a bounded domain in $\R^N,\, N \ge 2$. In a celebrated paper \cite{MR709644}, Br{\'e}zis and Nirenberg considered the problem
\begin{equation} \label{1.1}
\left\{\begin{aligned}
- \Delta u & = \lambda u + |u|^{2^\ast - 2}\, u && \text{in } \Omega\\[10pt]
u & = 0 && \text{on } \bdry{\Omega}
\end{aligned}\right.
\end{equation}
when $N \ge 3$, where $2^\ast = 2N/(N - 2)$ is the critical Sobolev exponent. Among other things, they proved that this problem has a positive solution when $N \ge 4$ and $0 < \lambda < \lambda_1$, where $\lambda_1 > 0$ is the first Dirichlet eigenvalue of $- \Delta$ in $\Omega$. Capozzi et al. \cite{MR831041} extended this result by proving the existence of a nontrivial solution when $N = 4$ and $\lambda > \lambda_1$ is not an eigenvalue, and when $N \ge 5$ and $\lambda \ge \lambda_1$. Garc{\'{\i}}a Azorero and Peral Alonso \cite{MR912211}, Egnell \cite{MR956567}, and Guedda and V{\'e}ron \cite{MR1009077} studied the corresponding problem for the $p$-Laplacian
\begin{equation} \label{1.2}
\left\{\begin{aligned}
- \Delta_p\, u & = \lambda\, |u|^{p-2}\, u + |u|^{p^\ast - 2}\, u && \text{in } \Omega\\[10pt]
u & = 0 && \text{on } \bdry{\Omega}
\end{aligned}\right.
\end{equation}
when $N > p > 1$, where $p^\ast = Np/(N - p)$. They proved that this problem has a positive solution when $N \ge p^2$ and $0 < \lambda < \lambda_1(p)$, where $\lambda_1(p) > 0$ is the first Dirichlet eigenvalue of $- \Delta_p$ in $\Omega$. Degiovanni and Lancelotti \cite{MR2514055} extended their result by proving the existence of a nontrivial solution when $N \ge p^2$ and $\lambda > \lambda_1(p)$ is not an eigenvalue, and when $N^2/(N + 1) > p^2$ and $\lambda \ge \lambda_1(p)$ (see also Arioli and Gazzola \cite{MR1741848}).

In the borderline case $N = p \ge 2$, the critical growth is of exponential type and is governed by the Trudinger-Moser inequality
\begin{equation} \label{1.3}
\sup_{u \in W^{1,N}_0(\Omega),\; \norm[N]{\nabla u} \le 1}\, \int_\Omega e^{\, \alpha_N\, |u|^{N'}} dx < \infty,
\end{equation}
where $W^{1,N}_0(\Omega)$ is the usual Sobolev space with the norm $\norm[N]{\nabla u} = \left(\int_\Omega |\nabla u|^N\, dx\right)^{1/N}$, $\alpha_N = N \omega_{N-1}^{1/(N-1)}$, $\omega_{N-1}$ is the area of the unit sphere in $\R^N$, and $N' = N/(N - 1)$ (see Trudinger \cite{MR0216286} and Moser \cite{MR0301504}). A natural analog of problem \eqref{1.2} for this case is
\begin{equation} \label{1.4}
\left\{\begin{aligned}
- \Delta_N\, u & = \lambda\, |u|^{N-2}\, ue^{\, |u|^{N'}} && \text{in } \Omega\\[10pt]
u & = 0 && \text{on } \bdry{\Omega},
\end{aligned}\right.
\end{equation}
where $\Delta_N\, u = \divg \left(|\nabla u|^{N-2}\, \nabla u\right)$ is the $N$-Laplacian of $u$. A result of Adimurthi \cite{MR1079983} implies that this problem has a nonnegative and nontrivial solution when $0 < \lambda < \lambda_1(N)$, where $\lambda_1(N) > 0$ is the first Dirichlet eigenvalue of $- \Delta_N$ in $\Omega$ (see also do {\'O} \cite{MR1392090}). Theorem 1.4 in de Figueiredo et al. \cite{MR1386960,MR1399846} implies that the semilinear problem
\begin{equation} \label{1.5}
\left\{\begin{aligned}
- \Delta u & = \lambda ue^{\, u^2} && \text{in } \Omega\\[10pt]
u & = 0 && \text{on } \bdry{\Omega}
\end{aligned}\right.
\end{equation}
has a nontrivial solution when $N = 2$ and $\lambda \ge \lambda_1$. In the present paper we first prove the existence of a nontrivial solution of problem \eqref{1.4} when $N \ge 3$ and $\lambda > \lambda_1(N)$ is not an eigenvalue. We have the following theorem.

\begin{theorem} \label{Theorem 1.1}
If $\lambda > 0$ is not a Dirichlet eigenvalue of $- \Delta_N$ in $\Omega$, then problem \eqref{1.4} has a nontrivial solution.
\end{theorem}

This extension to the quasilinear case is nontrivial. Indeed, the linking argument based on eigenspaces of $- \Delta$ in de Figueiredo et al. \cite{MR1386960,MR1399846} does not work when $N \ge 3$ since the nonlinear operator $- \Delta_N$ does not have linear eigenspaces. We will use a more general construction based on sublevel sets as in Perera and Szulkin \cite{MR2153141} (see also Perera et al.\! \cite[Proposition 3.23]{MR2640827}). Moreover, the standard sequence of eigenvalues of $- \Delta_N$ based on the genus does not give enough information about the structure of the sublevel sets to carry out this linking construction. Therefore we will use a different sequence of eigenvalues introduced in Perera \cite{MR1998432} that is based on a cohomological index, and show that problem \eqref{1.4} has a nontrivial solution if $\lambda > 0$ is not an eigenvalue from this particular sequence.

The $\Z_2$-cohomological index of Fadell and Rabinowitz \cite{MR57:17677} is defined as follows. Let $W$ be a Banach space and let $\A$ denote the class of symmetric subsets of $W \setminus \set{0}$. For $A \in \A$, let $\overline{A} = A/\Z_2$ be the quotient space of $A$ with each $u$ and $-u$ identified, let $f : \overline{A} \to \RP^\infty$ be the classifying map of $\overline{A}$, and let $f^\ast : H^\ast(\RP^\infty) \to H^\ast(\overline{A})$ be the induced homomorphism of the Alexander-Spanier cohomology rings. The cohomological index of $A$ is defined by
\[
i(A) = \begin{cases}
\sup \set{m \ge 1 : f^\ast(\omega^{m-1}) \ne 0}, & A \ne \emptyset\\[5pt]
0, & A = \emptyset,
\end{cases}
\]
where $\omega \in H^1(\RP^\infty)$ is the generator of the polynomial ring $H^\ast(\RP^\infty) = \Z_2[\omega]$. For example, the classifying map of the unit sphere $S^{m-1}$ in $\R^m,\, m \ge 1$ is the inclusion $\RP^{m-1} \incl \RP^\infty$, which induces isomorphisms on $H^q$ for $q \le m - 1$, so $i(S^{m-1}) = m$.

The Dirichlet spectrum of $- \Delta_N$ in $\Omega$ consists of those $\lambda \in \R$ for which the problem
\begin{equation} \label{1.6}
\left\{\begin{aligned}
- \Delta_N\, u & = \lambda\, |u|^{N-2}\, u && \text{in } \Omega\\[10pt]
u & = 0 && \text{on } \bdry{\Omega}
\end{aligned}\right.
\end{equation}
has a nontrivial solution. Although a complete description of the spectrum is not yet known when $N \ge 3$, we can define an increasing and unbounded sequence of eigenvalues via a suitable minimax scheme. The standard scheme based on the genus does not give the index information necessary to prove Theorem \ref{Theorem 1.1}, so we will use the following scheme based on the cohomological index as in Perera \cite{MR1998432}. Let
\begin{equation} \label{1.9}
\Psi(u) = \frac{1}{\dint_\Omega |u|^N\, dx}, \quad u \in \M = \set{u \in W^{1,N}_0(\Omega) : \int_\Omega |\nabla u|^N\, dx = 1}.
\end{equation}
Then eigenvalues of problem \eqref{1.6} on $\M$ coincide with critical values of $\Psi$. We use the standard notation
\[
\Psi^a = \set{u \in \M : \Psi(u) \le a}, \quad \Psi_a = \set{u \in \M : \Psi(u) \ge a}, \quad a \in \R
\]
for the sublevel sets and superlevel sets, respectively. Let $\F$ denote the class of symmetric subsets of $\M$ and set
\[
\lambda_k(N) := \inf_{M \in \F,\; i(M) \ge k}\, \sup_{u \in M}\, \Psi(u), \quad k \in \N.
\]
Then $0 < \lambda_1(N) < \lambda_2(N) \le \lambda_3(N) \le \cdots \to + \infty$ is a sequence of eigenvalues of problem \eqref{1.6} and
\begin{equation} \label{1.7}
\lambda_k(N) < \lambda_{k+1}(N) \implies i(\Psi^{\lambda_k(N)}) = i(\M \setminus \Psi_{\lambda_{k+1}(N)}) = k
\end{equation}
(see Perera et al.\! \cite[Propositions 3.52 and 3.53]{MR2640827}). Proof of Theorem \ref{Theorem 1.1} will make essential use of \eqref{1.7}.

Now we turn to the question of multiplicity of solutions to problem \eqref{1.4}. Let $0 < \lambda_1 < \lambda_2 \le \lambda_3 \le \cdots \to + \infty$ be the Dirichlet eigenvalues of $- \Delta$ in $\Omega$, repeated according to multiplicity, let
\[
S = \inf_{u \in H^1_0(\Omega) \setminus \set{0}}\, \frac{\norm[2]{\nabla u}^2}{\norm[2^\ast]{u}^2}
\]
be the best constant for the Sobolev imbedding $H^1_0(\Omega) \hookrightarrow L^{2^\ast}(\Omega)$ when $N \ge 3$, and let $\vol{\cdot}$ denote the Lebesgue measure in $\R^N$. Cerami et al.\! \cite{MR779872} proved that if $\lambda_k \le \lambda < \lambda_{k+1}$ and
\[
\lambda > \lambda_{k+1} - \frac{S}{\vol{\Omega}^{2/N}},
\]
and $m$ denotes the multiplicity of $\lambda_{k+1}$, then problem \eqref{1.1} has $m$ distinct pairs of nontrivial solutions $\pm\, u^\lambda_j,\, j = 1,\dots,m$ such that $u^\lambda_j \to 0$ as $\lambda \nearrow \lambda_{k+1}$. A result of Adimurthi and Yadava \cite{MR1044289} implies that there exists a constant $\mu_k \in [\lambda_k,\lambda_{k+1})$ such that if $\mu_k < \lambda < \lambda_{k+1}$, then the same conclusion holds for problem \eqref{1.5} when $N = 2$. We prove a similar bifurcation result for problem \eqref{1.4} when $N \ge 3$. We have the following theorem.

\begin{theorem} \label{Theorem 1.2}
If $N \ge 3$, $\lambda_k(N) < \lambda < \lambda_{k+1}(N) = \cdots = \lambda_{k+m}(N)$ for some $k, m \in \N$, and
\begin{equation} \label{1.8}
\lambda > \lambda_{k+1}(N) - \left(\frac{N \alpha_N^{N-1}}{\vol{\Omega}}\right)^{1/N}\! \lambda_k(N)^{1/N'},
\end{equation}
then problem \eqref{1.4} has $m$ distinct pairs of nontrivial solutions $\pm\, u^\lambda_j,\, j = 1,\dots,m$ such that $u^\lambda_j \to 0$ as $\lambda \nearrow \lambda_{k+1}(N)$.
\end{theorem}

The abstract result of Bartolo et al. \cite{MR713209} used in Cerami et al.\! \cite{MR779872} and Adimurthi and Yadava \cite{MR1044289} is based on linear subspaces and therefore cannot be used to prove Theorem \ref{Theorem 1.2}. We will prove a more general critical point theorem based on a pseudo-index related to the cohomological index that is applicable here (see also Perera et al.\! \cite[Proposition 3.44]{MR2640827}).

In closing the introduction we remark that we have confined ourselves to the model problem \eqref{1.4} only for the sake of simplicity. The methods developed in this paper can easily be adapted to treat nonlinearities more general than $|u|^{N-2}\, ue^{\, |u|^{N'}}$ as in Adimurthi \cite{MR1079983}, Adimurthi and Yadava \cite{MR1044289}, de Figueiredo et al. \cite{MR1386960,MR1399846}, and do {\'O} \cite{MR1392090}.

\section{Abstract critical point theorems}

In this section we prove two abstract critical point theorems based on the cohomological index that we will use to prove Theorems \ref{Theorem 1.1} and \ref{Theorem 1.2}. The following proposition summarizes the basic properties of the cohomological index.

\begin{proposition}[Fadell-Rabinowitz \cite{MR57:17677}] \label{Proposition 2.1}
The index $i : \A \to \N \cup \set{0,\infty}$ has the following properties:
\begin{properties}{i}
\item Definiteness: $i(A) = 0$ if and only if $A = \emptyset$;
\item \label{i2} Monotonicity: If there is an odd continuous map from $A$ to $B$ (in particular, if $A \subset B$), then $i(A) \le i(B)$. Thus, equality holds when the map is an odd homeomorphism;
\item Dimension: $i(A) \le \dim W$;
\item Continuity: If $A$ is closed, then there is a closed neighborhood $N \in \A$ of $A$ such that $i(N) = i(A)$. When $A$ is compact, $N$ may be chosen to be a $\delta$-neighborhood $N_\delta(A) = \set{u \in W : \dist{u}{A} \le \delta}$;
\item Subadditivity: If $A$ and $B$ are closed, then $i(A \cup B) \le i(A) + i(B)$;
\item \label{i6} Stability: If $SA$ is the suspension of $A \ne \emptyset$, obtained as the quotient space of $A \times [-1,1]$ with $A \times \set{1}$ and $A \times \set{-1}$ collapsed to different points, then $i(SA) = i(A) + 1$;
\item \label{i7} Piercing property: If $A$, $A_0$ and $A_1$ are closed, and $\varphi : A \times [0,1] \to A_0 \cup A_1$ is a continuous map such that $\varphi(-u,t) = - \varphi(u,t)$ for all $(u,t) \in A \times [0,1]$, $\varphi(A \times [0,1])$ is closed, $\varphi(A \times \set{0}) \subset A_0$ and $\varphi(A \times \set{1}) \subset A_1$, then $i(\varphi(A \times [0,1]) \cap A_0 \cap A_1) \ge i(A)$;
\item Neighborhood of zero: If $U$ is a bounded closed symmetric neighborhood of $0$, then $i(\bdry{U}) = \dim W$.
\end{properties}
\end{proposition}

Let
\[
S = \set{u \in W : \norm{u} = 1}
\]
be the unit sphere in $W$ and let
\[
\pi : W \setminus \set{0} \to S, \quad u \mapsto \frac{u}{\norm{u}}
\]
be the radial projection onto $S$. The following abstract result generalizes the linking theorem of Rabinowitz \cite{MR0488128}.

\begin{theorem} \label{Theorem 2.2}
Let $\Phi$ be a $C^1$-functional on $W$ and let $A_0,\, B_0$ be disjoint nonempty closed symmetric subsets of $S$ such that
\begin{equation} \label{2.1}
i(A_0) = i(S \setminus B_0) < \infty.
\end{equation}
Assume that there exist $R > r > 0$ and $v \in S \setminus A_0$ such that
\[
\sup \Phi(A) \le \inf \Phi(B), \qquad \sup \Phi(X) < \infty,
\]
where
\begin{gather*}
A = \set{tu : u \in A_0,\, 0 \le t \le R} \cup \set{R\, \pi((1 - t)\, u + tv) : u \in A_0,\, 0 \le t \le 1},\\[5pt]
B = \set{ru : u \in B_0},\\[5pt]
X = \set{tu : u \in A,\, \norm{u} = R,\, 0 \le t \le 1}.
\end{gather*}
Let $\Gamma = \set{\gamma \in C(X,W) : \gamma(X) \text{ is closed and} \restr{\gamma}{A} = \id{A}}$ and set
\[
c := \inf_{\gamma \in \Gamma}\, \sup_{u \in \gamma(X)}\, \Phi(u).
\]
Then
\begin{equation} \label{2.2}
\inf \Phi(B) \le c \le \sup \Phi(X),
\end{equation}
in particular, $c$ is finite. If, in addition, $\Phi$ satisfies the {\em \PS{c}} condition, then $c$ is a critical value of $\Phi$.
\end{theorem}

\begin{proof}
First we show that $A$ (homotopically) links $B$ with respect to $X$ in the sense that
\begin{equation} \label{2.3}
\gamma(X) \cap B \ne \emptyset \quad \forall \gamma \in \Gamma.
\end{equation}
If \eqref{2.3} does not hold, then there is a map $\gamma \in C(X,W \setminus B)$ such that $\gamma(X)$ is closed and $\restr{\gamma}{A} = \id{A}$. Let
\[
\widetilde{A} = \set{R\, \pi((1 - |t|)\, u + tv) : u \in A_0,\, -1 \le t \le 1}
\]
and note that $\widetilde{A}$ is closed since $A_0$ is closed (here $(1 - |t|)\, u + tv \ne 0$ since $v$ is not in the symmetric set $A_0$). Since
\[
SA_0 \to \widetilde{A}, \quad (u,t) \mapsto R\, \pi((1 - |t|)\, u + tv)
\]
is an odd continuous map,
\begin{equation} \label{2.4}
i(\widetilde{A}) \ge i(SA_0) = i(A_0) + 1
\end{equation}
by \ref{i2} and \ref{i6} of Proposition \ref{Proposition 2.1}. Consider the map
\[
\varphi : \widetilde{A} \times [0,1] \to W \setminus B, \quad \varphi(u,t) = \begin{cases}
\gamma(tu), & u \in \widetilde{A} \cap A\\[5pt]
- \gamma(-tu), & u \in \widetilde{A} \setminus A,
\end{cases}
\]
which is continuous since $\gamma$ is the identity on the symmetric set $\set{tu : u \in A_0,\, 0 \le t \le R}$. We have $\varphi(-u,t) = - \varphi(u,t)$ for all $(u,t) \in \widetilde{A} \times [0,1]$, $\varphi(\widetilde{A} \times [0,1]) = \gamma(X) \cup - \gamma(X)$ is closed, and $\varphi(\widetilde{A} \times \set{0}) = \set{0}$ and $\varphi(\widetilde{A} \times \set{1}) = \widetilde{A}$ since $\restr{\gamma}{A} = \id{A}$. Applying \ref{i7} with $\widetilde{A}_0 = \set{u \in W : \norm{u} \le r}$ and $\widetilde{A}_1 = \set{u \in W : \norm{u} \ge r}$ gives
\begin{equation} \label{2.5}
i(\widetilde{A}) \le i(\varphi(\widetilde{A} \times [0,1]) \cap \widetilde{A}_0 \cap \widetilde{A}_1) \le i((W \setminus B) \cap S_r) = i(S_r \setminus B) = i(S \setminus B_0),
\end{equation}
where $S_r = \set{u \in W : \norm{u} = r}$. By \eqref{2.4} and \eqref{2.5}, $i(A_0) < i(S \setminus B_0)$, contradicting \eqref{2.1}. Hence \eqref{2.3} holds.

It follows from \eqref{2.3} that $c \ge \inf \Phi(B)$, and $c \le \sup \Phi(X)$ since $\id{X} \in \Gamma$. If $\Phi$ satisfies the \PS{c} condition, then $c$ is a critical value of $\Phi$ by the classical minimax principle (see, e.g., Perera et al.\! \cite{MR2640827}).
\end{proof}

\begin{remark}
The linking construction in the proof of Theorem \ref{Theorem 2.2} was used in Perera and Szulkin \cite{MR2153141} to obtain nontrivial solutions of $p$-Laplacian problems with nonlinearities that interact with the spectrum. A similar construction based on the notion of cohomological linking was given in Degiovanni and Lancelotti \cite{MR2371112}. See also Perera et al.\! \cite[Proposition 3.23]{MR2640827}.
\end{remark}

Now let $\Phi$ be an even $C^1$-functional on $W$ and let $\A^\ast$ denote the class of symmetric subsets of $W$. Let $r > 0$, let $S_r = \set{u \in W : \norm{u} = r}$, let $0 < b \le + \infty$, and let $\Gamma$ denote the group of odd homeomorphisms of $W$ that are the identity outside $\Phi^{-1}(0,b)$. The pseudo-index of $M \in \A^\ast$ related to $i$, $S_r$, and $\Gamma$ is defined by
\[
i^\ast(M) = \min_{\gamma \in \Gamma}\, i(\gamma(M) \cap S_r)
\]
(see Benci \cite{MR84c:58014}). The following critical point theorem generalizes Bartolo et al. \cite[Theorem 2.4]{MR713209}.

\begin{theorem} \label{Theorem 2.4}
Let $A_0,\, B_0$ be symmetric subsets of $S$ such that $A_0$ is compact, $B_0$ is closed, and
\[
i(A_0) \ge k + m, \qquad i(S \setminus B_0) \le k
\]
for some $k, m \in \N$. Assume that there exists $R > r$ such that
\[
\sup \Phi(A) \le 0 < \inf \Phi(B), \qquad \sup \Phi(X) < b,
\]
where $A = \set{Ru : u \in A_0}$, $B = \set{ru : u \in B_0}$, and $X = \set{tu : u \in A,\, 0 \le t \le 1}$. For $j = k + 1,\dots,k + m$, let
\[
\A_j^\ast = \set{M \in \A^\ast : M \text{ is compact and } i^\ast(M) \ge j}
\]
and set
\[
c_j^\ast := \inf_{M \in \A_j^\ast}\, \max_{u \in M}\, \Phi(u).
\]
Then
\[
\inf \Phi(B) \le c_{k+1}^\ast \le \dotsb \le c_{k+m}^\ast \le \sup \Phi(X),
\]
in particular, $0 < c_j^\ast < b$. If, in addition, $\Phi$ satisfies the {\em \PS{c}} condition for all $c \in (0,b)$, then each $c_j^\ast$ is a critical value of $\Phi$ and there are $m$ distinct pairs of associated critical points.
\end{theorem}

\begin{proof}
If $M \in \A_{k+1}^\ast$,
\[
i(S_r \setminus B) = i(S \setminus B_0) \le k < k + 1 \le i^\ast(M) \le i(M \cap S_r)
\]
since $\id{W} \in \Gamma$. Hence $M$ intersects $B$ by \ref{i2} of Proposition \ref{Proposition 2.1}. It follows that $c_{k+1}^\ast \ge \inf \Phi(B)$.

If $\gamma \in \Gamma$, consider the continuous map
\[
\varphi : A \times [0,1] \to W, \quad \varphi(u,t) = \gamma(tu).
\]
We have $\varphi(A \times [0,1]) = \gamma(X)$, which is compact. Since $\gamma$ is odd, $\varphi(-u,t) = - \varphi(u,t)$ for all $(u,t) \in A \times [0,1]$ and $\varphi(A \times \set{0}) = \set{\gamma(0)} = \set{0}$. Since $\Phi \le 0$ on $A$, $\restr{\gamma}{A} = \id{A}$ and hence $\varphi(A \times \set{1}) = A$. Applying \ref{i7} with $\widetilde{A}_0 = \set{u \in W : \norm{u} \le r}$ and $\widetilde{A}_1 = \set{u \in W : \norm{u} \ge r}$ gives
\[
i(\gamma(X) \cap S_r) = i(\varphi(A \times [0,1]) \cap \widetilde{A}_0 \cap \widetilde{A}_1) \ge i(A) = i(A_0) \ge k + m.
\]
It follows that $i^\ast(X) \ge k + m$. So $X \in \A_{k+m}^\ast$ and hence $c_{k+m}^\ast \le \sup \Phi(X)$.

The rest now follows from standard results in critical point theory (see, e.g., Perera et al.\! \cite{MR2640827}).
\end{proof}

\begin{remark}
A similar construction was used in Perera and Szulkin \cite{MR2153141}. See also Perera et al.\! \cite[Proposition 3.44]{MR2640827}.
\end{remark}

\section{Variational setting}

Solutions of problem \eqref{1.4} coincide with critical points of the $C^1$-functional
\[
\Phi(u) = \int_\Omega \left[\frac{1}{N}\, |\nabla u|^N - \lambda\, F(u)\right] dx, \quad u \in W^{1,N}_0(\Omega),
\]
where
\[
F(t) = \int_0^t |s|^{N-2}\, se^{\, |s|^{N'}} ds = \int_0^{|t|} s^{N-1} e^{\, s^{N'}} ds.
\]
The following lemma is a special case of a result of Adimurthi \cite{MR1079983}.

\begin{lemma} \label{Lemma 3.1}
$\Phi$ satisfies the {\em \PS{c}} condition for all $c < \alpha_N^{N-1}/N$.
\end{lemma}

Let $\M$ and $\Psi$ be as in \eqref{1.9}. The following lemma implies that for any subset $A$ of $\M$ on which $\Psi$ is bounded, there exists $R > 0$ such that $\Phi(tu) \le 0$ for all $u \in A$ and $t \ge R$.

\begin{lemma} \label{Lemma 3.3}
For all $u \in \M$ and $t \ge 0$,
\[
\Phi(tu) \le \frac{t^N}{N} \left[1 - \frac{\lambda}{N' \vol{\Omega}^{1/(N-1)}} \left(\frac{t}{\Psi(u)}\right)^{N'}\right].
\]
\end{lemma}

\begin{proof}
Since $e^t \ge t$ for all $t \ge 0$,
\[
F(t) \ge \frac{|t|^{N+N'}}{N + N'} \quad \forall t \in \R,
\]
so
\[
\Phi(tu) \le t^N \left(\frac{1}{N} - \frac{\lambda t^{N'}}{N + N'} \int_\Omega |u|^{N+N'}\, dx\right).
\]
By the H\"older inequality,
\[
\vol{\Omega}^{1/(N-1)} \int_\Omega |u|^{N+N'}\, dx \ge \left(\int_\Omega |u|^N\, dx\right)^{N'} = \frac{1}{\Psi(u)^{N'}}. \QED
\]
\end{proof}

\section{Proof of Theorem \ref{Theorem 1.1}}

In this section we prove Theorem \ref{Theorem 1.1}. Our strategy is to apply Theorem \ref{Theorem 2.2} with suitable sets defined in terms of the eigenvalues of $- \Delta_N$, for which the minimax level $c$ is below the threshold for compactness given by Lemma \ref{Lemma 3.1}.

Since problem \eqref{1.4} has a nontrivial solution when $0 < \lambda < \lambda_1(N)$ by Adimurthi \cite{MR1079983}, we may assume that $\lambda > \lambda_1(N)$. Then
\begin{equation} \label{3.1}
\lambda_k(N) < \lambda < \lambda_{k+1}(N)
\end{equation}
for some $k$. By Degiovanni and Lancelotti \cite[Theorem 2.3]{MR2514055}, the sublevel set $\Psi^{\lambda_k(N)}$ has a compact symmetric subset $C$ of index $k$ that is bounded in $L^\infty(\Omega) \cap C^{1,\alpha}_\loc(\Omega)$. Without loss of generality we may assume that $0 \in \Omega$. For all $m \in \N$ so large that $B_{2/m}(0) \subset \Omega$, let
\[
\eta_m(x) = \begin{cases}
0, & |x| \le 1/2\, m^{m+1}\\[5pt]
2\, m^m \left(|x| - \dfrac{1}{2\, m^{m+1}}\right), & 1/2\, m^{m+1} < |x| \le 1/m^{m+1}\\[10pt]
(m\, |x|)^{1/m}, & 1/m^{m+1} < |x| \le 1/m\\[5pt]
1, & |x| > 1/m
\end{cases}
\]
(see Zhang et al. \cite{MR2112476}). Let
\[
\pi(u) = \frac{u}{\norm{u}}, \quad u \in W^{1,N}_0(\Omega) \setminus \set{0}
\]
be the radial projection onto $\M$.

\begin{lemma}
As $m \to \infty$,
\begin{gather}
\label{3.2} \int_\Omega |\eta_m u|^N\, dx = \int_\Omega |u|^N\, dx + \O\left(\frac{1}{m^N}\right);\\[10pt]
\label{3.3} \int_\Omega |\nabla (\eta_m u)|^N\, dx = 1 + \O\left(\frac{1}{m^{N-1}}\right);\\[10pt]
\label{3.4} \Psi(\pi(\eta_m u)) = \Psi(u) + \O\left(\frac{1}{m^{N-1}}\right)
\end{gather}
uniformly in $u \in C$.
\end{lemma}

\begin{proof}
We have
\[
\abs{\int_\Omega |\eta_m u|^N\, dx - \int_\Omega |u|^N\, dx} \le \int_{B_{1/m}(0)} \left(|\eta_m u|^N + |u|^N\right) dx = \O\left(\frac{1}{m^N}\right)
\]
since $C$ is bounded in $L^\infty(\Omega)$ and $|\eta_m| \le 1$, so \eqref{3.2} holds. Next
\[
\abs{\int_\Omega |\nabla (\eta_m u)|^N\, dx - \int_\Omega |\nabla u|^N\, dx} \le \int_{B_{1/m}(0)} \left(|\nabla (\eta_m u)|^N + |\nabla u|^N\right) dx
\]
and
\[
\int_{B_{1/m}(0)} |\nabla (\eta_m u)|^N\, dx \le \sum_{j=0}^N \binom{N}{j} \int_{B_{1/m}(0)} |\nabla u|^{N-j}\, |u|^j\, |\nabla \eta_m|^j\, dx.
\]
Since $C$ is bounded in $C^1(B_{1/m}(0))$, $u$ and $\nabla u$ are bounded, and a direct calculation shows that
\[
\int_{B_{1/m}(0)} |\nabla \eta_m|^j\, dx = \O\left(\frac{1}{m^{N-1}}\right), \quad j = 0,\dots,N,
\]
so \eqref{3.3} follows. Since
\[
\Psi(\pi(\eta_m u)) = \frac{\dint_\Omega |\nabla (\eta_m u)|^N\, dx}{\dint_\Omega |\eta_m u|^N\, dx},
\]
\eqref{3.4} is immediate from \eqref{3.2} and \eqref{3.3}.
\end{proof}

Set $C_m = \set{\pi(\eta_m u) : u \in C}$. Since $C \subset \Psi^{\lambda_k(N)}$,
\[
\Psi(\pi(\eta_m u)) \le \lambda_k(N) + \O\left(\frac{1}{m^{N-1}}\right) \quad \forall u \in C
\]
by \eqref{3.4}. Using $\lambda_k(N) < \lambda$, we fix $m$ so large that
\begin{equation} \label{3.5}
\Psi(u) \le \lambda \quad \forall u \in C_m.
\end{equation}
Then $C_m \subset \M \setminus \Psi_{\lambda_{k+1}(N)}$ since $\lambda < \lambda_{k+1}(N)$, so
\[
i(C_m) \le i(\M \setminus \Psi_{\lambda_{k+1}(N)}) = k
\]
by \ref{i2} of Proposition \ref{Proposition 2.1} and \eqref{1.7}. On the other hand, $C \to C_m,\, u \mapsto \pi(\eta_m u)$ is an odd continuous map and hence
\[
i(C_m) \ge i(C) = k
\]
by \ref{i2} again. Thus,
\begin{equation} \label{3.6}
i(C_m) = k.
\end{equation}

We are now ready to prove Theorem \ref{Theorem 1.1}.

\begin{proof}[Proof of Theorem \ref{Theorem 1.1}]
We apply Theorem \ref{Theorem 2.2} to our functional $\Phi$ with
\[
A_0 = C_m, \qquad B_0 = \Psi_{\lambda_{k+1}(N)},
\]
noting that \eqref{2.1} follows from \eqref{3.6}, \eqref{3.1}, and \eqref{1.7}. Let $R > r > 0$, let $v \in \M \setminus C_m$, and let $A$, $B$ and $X$ be as in Theorem \ref{Theorem 2.2}.

First we show that $\inf \Phi(B) > 0$ if $r$ is sufficiently small. Since $e^t \le 1 + te^t$ for all $t \ge 0$,
\[
F(t) \le \frac{|t|^N}{N} + |t|^{N+N'} e^{\, |t|^{N'}} \quad \forall t \in \R,
\]
so for $u \in \Psi_{\lambda_{k+1}(N)}$,
\begin{eqnarray*}
\Phi(ru) & \ge & \int_\Omega \left[\frac{r^N}{N}\, |\nabla u|^N - \frac{\lambda r^N}{N}\, |u|^N - \lambda r^{N+N'}\, |u|^{N+N'} e^{\, r^{N'}\! |u|^{N'}}\right] dx\\[10pt]
& \ge & \frac{r^N}{N} \left(1 - \frac{\lambda}{\lambda_{k+1}(N)}\right) - \lambda r^{N+N'}\, \bigg(\int_\Omega e^{\, 2r^{N'}\! |u|^{N'}} dx\bigg)^{1/2} \norm[2\, (N+N')]{u}^{N+N'}.
\end{eqnarray*}
If $2\, r^{N'} \le \alpha_N$, then
\[
\int_\Omega e^{\, 2r^{N'}\! |u|^{N'}} dx \le \int_\Omega e^{\, \alpha_N |u|^{N'}} dx,
\]
which is bounded by \eqref{1.3}. Since $W^{1,N}_0(\Omega) \hookrightarrow L^{2\, (N+N')}(\Omega)$ and $\lambda < \lambda_{k+1}(N)$, it follows that $\inf \Phi(B) > 0$ if $r$ is sufficiently small.

Next we show that $\sup \Phi(A) \le 0$ if $R$ is sufficiently large. Since $e^t \ge 1$ for all $t \ge 0$,
\[
F(t) \ge \frac{|t|^N}{N} \quad \forall t \in \R,
\]
so for $u \in C_m$ and any $t \ge 0$,
\begin{eqnarray*}
\Phi(tu) & \le & \int_\Omega \left[\frac{t^N}{N}\, |\nabla u|^N - \frac{\lambda t^N}{N}\, |u|^N\right] dx\\[10pt]
& \le & \frac{t^N}{N} \left(1 - \frac{\lambda}{\Psi(u)}\right)\\[7.5pt]
& \le & 0
\end{eqnarray*}
by \eqref{3.5}. Since $C$ is compact and the map $C \to C_m,\, u \mapsto \pi(\eta_m u)$ is continuous, $C_m$ is compact, and hence so is the set $\set{\pi((1 - t)\, u + tv) : u \in C_m,\, 0 \le t \le 1}$. So $\Psi$ is bounded on this set, and there exists $R > r$ such that $\Phi \le 0$ on $\set{R\, \pi((1 - t)\, u + tv) : u \in C_m,\, 0 \le t \le 1}$ by Lemma \ref{Lemma 3.3}.

Now we show that $\sup \Phi(X) < \alpha_N^{N-1}/N$ for a suitably chosen $v$. Let
\[
v_j(x) = \frac{1}{\omega_{N-1}^{1/N}}\, \begin{cases}
(\log j)^{(N-1)/N}, & |x| \le 1/j\\[10pt]
\dfrac{\log |x|^{-1}}{(\log j)^{1/N}}, & 1/j < |x| \le 1\\[10pt]
0, & |x| > 1.
\end{cases}
\]
Then $v_j \in W^{1,N}(\R^N)$, $\norm[N]{\nabla v_j} = 1$, and $\norm[N]{v_j}^N = \O(1/\log j)$ as $j \to \infty$. We take $v(x) = \widetilde{v}_j(x) := v_j(x/r_m)$ with $r_m = 1/2\, m^{m+1}$ and $j$ sufficiently large. Since $B_{r_m}(0) \subset \Omega$, $\widetilde{v}_j \in W^{1,N}_0(\Omega)$ and $\norm[N]{\nabla \widetilde{v}_j} = 1$. For sufficiently large $j$,
\[
\Psi(\widetilde{v}_j) = \frac{1}{r_m^N\, \norm[N]{v_j}^N} > \lambda
\]
and hence $\widetilde{v}_j \notin C_m$ by \eqref{3.5}. For $u \in C_m$ and $s,\, t \ge 0$,
\[
\Phi(su + t \widetilde{v}_j) = \Phi(su) + \Phi(t \widetilde{v}_j)
\]
since $u = 0$ on $B_{r_m}(0)$ and $\widetilde{v}_j = 0$ on $\Omega \setminus B_{r_m}(0)$. Since $\Phi(su) \le 0$, it suffices to show that $\sup_{t \ge 0}\, \Phi(t \widetilde{v}_j) < \alpha_N^{N-1}/N$ for arbitrarily large $j$. Since $\Phi(t \widetilde{v}_j) \to - \infty$ as $t \to + \infty$ by Lemma \ref{Lemma 3.3}, there exists $t_j \ge 0$ such that
\begin{equation} \label{4.1T}
\Phi(t_j \widetilde{v}_j) = \frac{t_j^N}{N} - \lambda \int_{B_{r_m}(0)} F(t_j \widetilde{v}_j)\, dx = \sup_{t \ge 0}\, \Phi(t \widetilde{v}_j)
\end{equation}
and
\begin{equation} \label{4.2T}
\Phi'(t_j \widetilde{v}_j)\, \widetilde{v}_j = t_j^{N-1} \left(1 - \lambda \int_{B_{r_m}(0)} \widetilde{v}_j^N e^{\, t_j^{N'} \widetilde{v}_j^{N'}}\, dx\right) = 0.
\end{equation}
Suppose $\Phi(t_j \widetilde{v}_j) \ge \alpha_N^{N-1}/N$ for all sufficiently large $j$. Since $F(t) \ge 0$ for all $t \in \R$, then \eqref{4.1T} gives $t_j^{N'} \ge \alpha_N$, and then \eqref{4.2T} gives
\begin{multline*}
\frac{1}{\lambda} = \int_{B_{r_m}(0)} \widetilde{v}_j^N e^{\, t_j^{N'} \widetilde{v}_j^{N'}}\, dx \ge \int_{B_{r_m}(0)} \widetilde{v}_j^N e^{\, \alpha_N\, \widetilde{v}_j^{N'}}\, dx\\[10pt]
= r_m^N \int_{B_1(0)} v_j^N e^{\, \alpha_N\, v_j^{N'}}\, dx \ge r_m^N \int_{B_{1/j}(0)} v_j^N e^{\, \alpha_N\, v_j^{N'}}\, dx = \frac{r_m^N}{N}\, (\log j)^{N-1},
\end{multline*}
which is impossible for large $j$.

Now
\[
c \le \sup \Phi(X) < \frac{\alpha_N^{N-1}}{N}
\]
by \eqref{2.2}, so $\Phi$ satisfies the \PS{c} condition by Lemma \ref{Lemma 3.1}. Thus, $\Phi$ has a critical point $u$ at the level $c$ by Theorem \ref{Theorem 2.2}. Since
\[
c \ge \inf \Phi(B) > 0
\]
by \eqref{2.2} again, $u$ is nontrivial.
\end{proof}

\section{Proof of Theorem \ref{Theorem 1.2}}

In this section we prove Theorem \ref{Theorem 1.2}.

\begin{lemma}
For all $t \in \R$,
\begin{gather}
\label{4.2} F(t) \le \frac{|t|^N}{N}\; e^{\, |t|^{N'}} - \frac{|t|^{N+N'}}{N^2};\\[5pt]
\label{4.3} F(t) \le \frac{|t|^{N-N'}}{N'}\; e^{\, |t|^{N'}}.
\end{gather}
\end{lemma}

\begin{proof}
Integrating by parts gives
\begin{multline*}
F(t) = \frac{|t|^N}{N}\; e^{\, |t|^{N'}} - \frac{N'}{N} \int_0^{|t|} s^{N+N'-1} e^{\, s^{N'}} ds\\[10pt]
\le \frac{|t|^N}{N}\; e^{\, |t|^{N'}} - \frac{N'}{N} \int_0^{|t|} s^{N+N'-1}\, ds = \frac{|t|^N}{N}\; e^{\, |t|^{N'}} - \frac{|t|^{N+N'}}{N^2}
\end{multline*}
and
\[
F(t) = \frac{|t|^{N-N'}}{N'}\; e^{\, |t|^{N'}} - \frac{N - N'}{N'} \int_0^{|t|} s^{N-N'-1} e^{\, s^{N'}} ds \le \frac{|t|^{N-N'}}{N'}\; e^{\, |t|^{N'}}. \QED
\]
\end{proof}

We are now ready to prove Theorem \ref{Theorem 1.2}.

\begin{proof}[Proof of Theorem \ref{Theorem 1.2}]
In view of Lemma \ref{Lemma 3.1}, we apply Theorem \ref{Theorem 2.4} with $b = \alpha_N^{N-1}/N$. By Degiovanni and Lancelotti \cite[Theorem 2.3]{MR2514055}, the sublevel set $\Psi^{\lambda_{k+m}(N)}$ has a compact symmetric subset $A_0$ with
\[
i(A_0) = k + m.
\]
We take $B_0 = \Psi_{\lambda_{k+1}(N)}$, so that
\[
i(S \setminus B_0) = k
\]
by \eqref{1.7}. Let $R > r > 0$ and let $A$, $B$ and $X$ be as in Theorem \ref{Theorem 2.4}. As in the proof of Theorem \ref{Theorem 1.1}, $\inf \Phi(B) > 0$ if $r$ is sufficiently small. Since $A_0 \subset \Psi^{\lambda_{k+1}(N)}$, there exists $R > r$ such that $\Phi \le 0$ on $A$ by Lemma \ref{Lemma 3.3}. Since $e^t \ge 1 + t$ for all $t \ge 0$,
\[
F(t) \ge \frac{|t|^N}{N} + \frac{|t|^{N+N'}}{N + N'} \quad \forall t \in \R,
\]
so for $u \in X$,
\begin{eqnarray*}
\Phi(u) & \le & \int_\Omega \left[\frac{1}{N}\, |\nabla u|^N - \frac{\lambda}{N}\, |u|^N - \frac{\lambda}{N + N'}\, |u|^{N+N'}\right] dx\\[10pt]
& \le & \frac{\lambda_{k+1}(N) - \lambda}{N} \int_\Omega |u|^N\, dx - \frac{\lambda_k(N)}{(N + N') \vol{\Omega}^{1/(N-1)}} \left(\int_\Omega |u|^N\, dx\right)^{N'}\\[10pt]
& \le & \sup_{\rho \ge 0}\, \left[\frac{(\lambda_{k+1}(N) - \lambda)\, \rho}{N} - \frac{\lambda_k(N)\, \rho^{N'}}{(N + N') \vol{\Omega}^{1/(N-1)}}\right]\\[10pt]
& = & \frac{(\lambda_{k+1}(N) - \lambda)^N \vol{\Omega}}{N^2\, \lambda_k(N)^{N-1}}.
\end{eqnarray*}
So
\[
\sup \Phi(X) \le \frac{(\lambda_{k+1}(N) - \lambda)^N \vol{\Omega}}{N^2\, \lambda_k(N)^{N-1}} < \frac{\alpha_N^{N-1}}{N}
\]
by \eqref{1.8}. Thus, problem \eqref{1.4} has $m$ distinct pairs of nontrivial solutions $\pm\, u^\lambda_j,\, j = 1,\dots,m$ such that
\begin{equation} \label{4.4}
0 < \Phi(u^\lambda_j) \le \frac{(\lambda_{k+1}(N) - \lambda)^N \vol{\Omega}}{N^2\, \lambda_k(N)^{N-1}}.
\end{equation}

To prove that $u^\lambda_j \to 0$ as $\lambda \nearrow \lambda_{k+1}(N)$, it suffices to show that for every sequence $\nu_n \nearrow \lambda_{k+1}(N)$, a subsequence of $v_n := u^{\nu_n}_j$ converges to zero. We have
\begin{equation} \label{4.5}
\Phi(v_n) = \int_\Omega \left[\frac{1}{N}\, |\nabla v_n|^N - \nu_n\, F(v_n)\right] dx \to 0
\end{equation}
by \eqref{4.4} and
\begin{equation} \label{4.6}
\Phi'(v_n)\, v_n = \int_\Omega \left[|\nabla v_n|^N - \nu_n\, |v_n|^N e^{\, |v_n|^{N'}}\right] dx = 0.
\end{equation}
By \eqref{4.2}, \eqref{4.5}, and \eqref{4.6},
\[
\frac{1}{N^2} \int_\Omega |v_n|^{N+N'}\, dx \le \int_\Omega \left[\frac{1}{N}\, |v_n|^N e^{\, |v_n|^{N'}} - F(v_n)\right] dx = \frac{\Phi(v_n)}{\nu_n} \le \frac{\Phi(v_n)}{\lambda_k(N)} \to 0,
\]
so $v_n \to 0$ a.e.\! in $\Omega$ for a renamed subsequence. By \eqref{4.3},
\begin{equation} \label{4.7}
N' \int_\Omega F(v_n)\, dx \le \int_\Omega |v_n|^{N-N'} e^{\, |v_n|^{N'}} dx =: I_1 + I_2,
\end{equation}
where
\begin{equation} \label{4.8}
I_1 = \int_{\set{|v_n| > (2N/N')^{1/N'}}} |v_n|^{N-N'} e^{\, |v_n|^{N'}} dx \le \frac{N'}{2N} \int_\Omega |v_n|^N e^{\, |v_n|^{N'}} dx = \frac{N'}{2N \nu_n}\, \norm{v_n}^N
\end{equation}
by \eqref{4.6} and
\begin{equation} \label{4.9}
I_2 = \int_\Omega \goodchi_{\set{|v_n| \le (2N/N')^{1/N'}}}\!(x)\, |v_n|^{N-N'} e^{\, |v_n|^{N'}} dx \to 0
\end{equation}
by the Lebesgue dominated convergence theorem. Combining \eqref{4.5}, \eqref{4.7}, and \eqref{4.8} gives
\[
\frac{1}{2N}\, \norm{v_n}^N \le \Phi(v_n) + \frac{\lambda_{k+1}(N)}{N'}\, I_2 \to 0
\]
by \eqref{4.5} and \eqref{4.9}.
\end{proof}

\def\cdprime{$''$}

\end{document}